\documentclass[reqno,12pt]{amsart}

\usepackage{enumerate}
\usepackage[margin=1in]{geometry}
\usepackage{ifpdf}
\usepackage{amsmath}
\usepackage{amsfonts}
\usepackage{amssymb}
\usepackage{amsthm}
\usepackage[hyperfootnotes=false]{hyperref}
\usepackage{setspace}
\usepackage{amsrefs}
\usepackage{nicefrac}

\newcommand{\ignore}[1]{}

\renewcommand{\Re}{\operatorname{Re}}
\renewcommand{\Im}{\operatorname{Im}}
\newcommand{\Orb}{\operatorname{Orb}}

\newcommand{\sing}{\operatorname{sing}}

\newcommand{\abs}[1]{\left\lvert {#1} \right\rvert}

\newcommand{\C}{{\mathbb{C}}}
\newcommand{\R}{{\mathbb{R}}}

\newcommand{\sF}{{\mathcal{F}}}

\newcommand{\sO}{{\mathcal{O}}}

\newtheorem{thm}{Theorem}[section]
\newtheorem*{thmnonum}{Theorem}

\newtheorem{prop}[thm]{Proposition}

\newtheorem{cor}[thm]{Corollary}

\newtheorem{lemma}[thm]{Lemma}

\theoremstyle{definition}
\newtheorem{defn}[thm]{Definition}

\theoremstyle{remark}

\author{Ji\v{r}\'{\i} Lebl}
\thanks{The author was in part supported by NSF grant DMS 0900885.}
\address{Department of Mathematics, University of California
at San Diego, La Jolla, CA 92093-0112, USA}
\curraddr{Department of Mathematics, University of Wisconsin, 
Madison, WI 53706, USA}
\email{jlebl@math.ucsd.edu}
\date{May 11, 2012}

\ifpdf
\hypersetup{
  pdftitle={Singular set of a Levi-flat hypersurface is Levi-flat},
  pdfauthor={Jiri Lebl}
}
\fi

\title{Singular set of a Levi-flat hypersurface is Levi-flat}

\begin{document}


\begin{abstract}
We study the singular set of a singular Levi-flat real-analytic
hypersurface.  We prove that the singular set of such a hypersurface
is Levi-flat in the appropriate sense.
We also show that if the singular set is small enough, then
the Levi-foliation extends to a singular codimension one holomorphic
foliation of a neighborhood of the hypersurface.
\end{abstract}

\maketitle



\section{Introduction} \label{section:intro}

A real smooth hypersurface $H$ in a complex manifold
is said to be \emph{Levi-flat} if the Levi-form vanishes
identically, or in other words
if it is pseudoconvex from both sides.
Levi-flat hypersurfaces occur naturally,
for example as invariant sets of holomorphic foliations.
A real-analytic nonsingular Levi-flat hypersurface is
locally biholomorphic to a hypersurface of the form
$\{ \Im z_1 = 0 \}$, and is therefore foliated by
complex hypersurfaces (called the \emph{Levi-foliation}).
The definition of Levi-flat can be naturally
extended to CR submanifolds of higher codimensions by requiring that
the Levi-form vanishes identically.  
A real-analytic CR manifold is 
\emph{Levi-flat} in this sense
if in suitable local coordinates we can write its defining equations
as $\Im z_1 = \cdots = \Im z_j = 0$ and
$z_{j+1} = \cdots = z_{k} = 0$ for some $j$ and $k$ (where we interpret
$j=0$ and $j=k$ in the obvious sense).  With this terminology we
consider complex manifolds to be Levi-flat.

In this article, we consider singular Levi-flat real-analytic
subvarieties.
Local questions about singular Levi-flat
hypersurfaces have been previously studied
by Bedford~\cite{Bedford:flat}, Burns and Gong~\cite{burnsgong:flat},
Fern\'andez-P\'erez~\cite{FernPer:norm},
and the author~\cites{Lebl:lfnm, Lebl:thesis}.
Real-algebraic singular Levi-flat hypersurfaces in complex
projective space when written in homogeneous coordinates
are real-algebraic Levi-flat complex cones and hence their classification
is a local question as well,
see \cite{Lebl:projlf}.
A natural and well studied question is how to divide the projective space
into pseudoconvex domains.  A well known theorem of Lins
Neto~\cite{linsneto:note} says that a Levi-flat hypersurface in projective
space is necessarily singular in dimension 3 and higher.  We therefore
need to understand the singular set of Levi-flat hypersurfaces.
See the books~\cites{BER:book, Krantz:book, DAngelo:CR}
for the basic language and background.

Let $U \subset \C^N$ be an open subset and let $H \subset U$
be a (closed) real-analytic subvariety of real dimension $2N-1$.
For simplicity,
we use the term \emph{real-hypervariety} for $H$.
Let $H^*$ be the set of points of $H$ near which $H$ is a nonsingular
real-analytic hypersurface.
We say $H$ is \emph{Levi-flat} if $H^*$ is Levi-flat.
Let $H_s$ be the set of singular points of $H$, points near which
$H$ is not a real-analytic submanifold (of any dimension).  The relative topological
closure $\overline{H^*} \cap U$
is a semianalytic set (a set defined by equalities and inequalities,
see~\cite{BM:semisub}),
and is the natural
object to study.  The singular set $(\overline{H^*} \cap U)_s$ is defined as
above.  It is easy to see that $(\overline{H^*} \cap U)_s \subset H_s$.
If $H = \overline{H^*} \cap U$, then $H_s = (\overline{H^*} \cap U)_s$.
Our main result is the following theorem.  We show that the
result is optimal given the hypothesis.

\newcommand{\mainthmbody}{
Let $U \subset \C^N$ be an open set and let
$H \subset U$ be a (closed)
Levi-flat real-hypervariety.
Then the singular set
$(\overline{H^*} \cap U)_s$ is Levi-flat near points where it is a CR
real-analytic submanifold.

Furthermore, if $(\overline{H^*} \cap U)_s$ is a generic submanifold, then
$(\overline{H^*} \cap U)_s$ is a
generic Levi-flat submanifold of dimension $2N-2$.
}
\begin{thm} \label{mainthm}
\mainthmbody
\end{thm}

A \emph{generic} real submanifold $M \subset \C^N$ is a submanifold 
with real defining equations $r_1(z,\bar{z}) = \cdots = r_k(z,\bar{z}) = 0$
such that
$\partial r_1, \ldots, \partial r_k$ are linearly independent.
Here
$\partial r = \sum \frac{\partial r}{\partial z_k} dz_k$ refers to the part
of the differential in the holomorphic variables.  In particular,
a generic submanifold is not contained in any proper complex variety.

The theorem is optimal in the sense that given simply the hypothesis that
$H^*$ is Levi-flat we cannot conclude that $H_s$ is Levi-flat; there could
be a lower dimensional component of $H$ which need not be Levi-flat.

Since a
semianalytic set is always contained in a real subvariety
of the same dimension,
the result also classifies singularities of
semianalytic Levi-flat hypersurfaces.

Burns and Gong~\cite{burnsgong:flat} construct many examples where the
singularity is a complex variety.  For example,
$\{ z \in \C^N : \Im (z_1^2+\cdots+z_k^2) = 0 \}$ is a Levi-flat
real-hypervariety with $\C^{N-k}$ as the singular set.

On the other hand, $\{ z : (\Im z_1)(\Im z_2) = 0 \}$
is a Levi-flat real-hypervariety with a generic Levi-flat singular set
$\{ z : \Im z_1 = \Im z_2 = 0 \}$.
It is possible to construct an irreducible
Levi-flat real-hypervariety with a generic singular set.
See Brunella~\cite{Brunella:lf} for an example.

We only study $(\overline{H^*} \cap U)_s$ near points where it is a 
real-analytic CR submanifold (a real-analytic submanifold is CR on an open
dense set).  It is possible that $(\overline{H^*} \cap U)_s$ is not a CR
submanifold.  For example the Levi-flat real-hypervariety
$\{ z : \bigl(\Re (z_2-z_1^2)\bigr)\bigl(\Im z_2\bigr) = 0 \}$
is a union of two
nonsingular Levi-flat hypersurfaces whose intersection is not a
CR submanifold at the origin.

As with
all real-analytic varieties, the singular set $H_s$
is not necessarily equal to
$H \setminus H^*$ even if $H$ is irreducible.  See~\cite{Lebl:projlf}
and~\cite{Brunella:lf} for examples of such Whitney-umbrella-type Levi-flat
hypervarieties.  The ``umbrella handle'' in those examples is
also generic Levi-flat or complex analytic.  The methods used in this
present article only give information on $\overline{H^*} \cap U$.
It is not known if an ``umbrella handle'' of an irreducible $H$ is
necessarily Levi-flat.
Note that while we know
that ${(\overline{H^*} \cap U)}_s \subset H_s$, the inclusion
could be proper even if the singular set is contained in
$\overline{H^*}$ as
points of $H_s$ may in fact be intersections of nonsingular points
of $\overline{H^*} \cap U$ with 
$H \setminus \overline{H^*}$.

Burns and Gong~\cite{burnsgong:flat}, and the
author~\cites{Lebl:lfnm,Lebl:projlf} also studied Levi-flat
real-hypervarieties defined
by $\Im f = 0$ for a holomorphic or a meromorphic function $f$.
Such Levi-flat real-hypervarieties have a complex analytic singular set,
but it turns out that not every Levi-flat hypersurface can be obtained
this way.
If a meromorphic function defined on a neighborhood of $H$
is constant on the leaves
of $H^*$, then 
the Levi-foliation extends to a possibly singular codimension one 
holomorphic foliation of a neighborhood of $H$.  That is,
near each point of $H$ there exists a nontrivial holomorphic 
one form $\omega$ that is completely integrable
($\omega \wedge d \omega = 0$)
and such that the leaves of the Levi-foliation of $H^*$
are integral manifolds of $\omega$ (tangent space of each
leaf is annihilated by $\omega$).
While near points of $H^*$, the foliation always extends, it is not
true that every Levi-flat real-hypervariety
is such that the foliation extends near singular points, even if $H$
is irreducible.  Brunella~\cite{Brunella:lf} proved that the foliation does
extend after lifting to the cotangent bundle.

In the proof of Theorem~\ref{mainthm} we must find sufficient conditions
for the foliation of $H^*$ to extend.  Besides proving what is needed
for Theorem~\ref{mainthm}, we have the following theorem that is
of independent interest.  Recently Cerveau and Lins
Neto~\cite{CerveauLinsNeto} proved a similar result.
By $H$ being \emph{leaf-degenerate} at $p \in H$,
we mean that there are infinitely many distinct 
germs of complex hypervarieties $(X,p) \subset (H,p)$,
see \S~\ref{section:ldeg} for a more precise definition.

\begin{thm} \label{thm:extsmallsing}
Let $U \subset \C^N$ be an open set and let
$H \subset U$ be a
Levi-flat real-hypervariety that
is irreducible as a germ at $p \in \overline{H^*} \cap U$.
If either
\begin{enumerate}[(i)]
\item $\dim H_s = 2N-4$ and $H$ is not leaf-degenerate at $p$, or
\item $\dim H_s < 2N-4$,
\end{enumerate}
then there exists a neighborhood $U'$ of $p$, and a nontrivial
holomorphic one form $\omega$ defined in $U'$, such that $\omega \wedge d
\omega  = 0$ and such that the leaves of the Levi-foliation
of $H^* \cap U'$ are integral submanifolds of $\omega$.
In other words, near $p$ the Levi-foliation extends to
a possibly singular
codimension one holomorphic foliation.
\end{thm}

A primary tool in the proofs is Lemma~\ref{dimnleaf:lemma}, which
says that through every point $p \in \overline{H^*} \cap U$ for a Levi-flat
real-hypervariety $H$, there exists a complex hypersurface $W$
such that $W \subset \overline{H^*} \cap U$.  This $W$ is generally
a branch of the Segre variety of $H$ at $p$, unless the Segre variety
is degenerate at $p$.  This lemma also implies that all sides
of $\overline{H^*} \cap U$ are pseudoconvex.  Hence a real-hypervariety
$H$ is Levi-flat
if and only if $\overline{H^*} \cap U$ is pseudoconvex from all sides.

The author would like to acknowledge Peter Ebenfelt for suggesting the study
of this problem when the author was still a graduate student, and also for
many conversations on the topic.  The author would also like to thank
Xianghong Gong, John P.\ D'Angelo, Salah Baouendi, Linda Rothschild, and
Arturo Fern\'andez-P\'erez
for fruitful discussions on topics
related to this research and suggestions on the manuscript.


\section{CR submanifolds}

Background for CR geometry is taken from the books
\cites{BER:book, Krantz:book, DAngelo:CR}.
For background on complex varieties
see the book \cite{Whitney:book}.

Let $M \subset \C^N$ be a real-analytic submanifold (not necessarily
closed) of dimension $n$.
We consider the complexified tangent space $\C \otimes T_p M$.
The tangent vectors of the form
$\sum_{j=1}^N a_j \frac{\partial}{\partial \bar{z}_j}$ tangent to $M$
are called the
\emph{CR vectors}.  If the space of CR vectors at $p$,
called $T^{0,1}_p M$, has
constant dimension at all points of $M$,
the submanifold is said to be a \emph{CR submanifold}.  The
complex dimension of $T^{0,1}_p M$ is called 
the \emph{CR dimension} of $M$.  If a
CR submanifold is not contained in a proper complex analytic subvariety,
it is a generic submanifold.  In fact a generic submanifold is automatically
CR.

Let $\Orb_p(M)$ denote the \emph{local CR orbit} of $M$ at $p$,
that is, the integral manifold of the distribution of CR vector fields
and all of their commutators.  For real-analytic $M$
the CR orbit is
guaranteed to exist by a theorem of Nagano,
and $\Orb_p(M)$  is
the germ of a CR submanifold of $M$ through $p$ of smallest
dimension that has the same CR dimension as $M$ (see \cite{BER:book}).

If $\Orb_p(M) = (M,p)$ as germs, then
$M$ is said to be \emph{minimal} at $p$.  If a real-analytic submanifold is
minimal at one point, then it is minimal outside a real-analytic subvariety
(again see \cite{BER:book}).

For a connected
real-analytic CR submanifold $M$, we find that $\Orb_p(M)$ attains
a maximal dimension for $p$ in a dense open subset of $M$.  Near a point
where the dimension of $\Orb_p(M)$ is maximal we have the following
well known theorem.

\begin{thm}[see Baouendi-Ebenfelt-Rothschild~\cite{BER:craut}]
\label{nicenormcoord}
Let $M \subset \C^N$ be a
real-analytic 
CR submanifold. 
Let $p \in M$ be such that $\Orb_p(M)$ is of maximal dimension.
Then there are coordinates
$(z,w,w',w'') \in \C^n \times \C^{d-q} \times \C^q \times \C^k = \C^N$,
vanishing at $p$,
where $k$ denotes the complex dimension of the intrinsic complexification of
$M$ near $p$, $d$ is the real codimension of $M$ in its intrinsic
complexification, and 
$q$ denotes the real codimension of $\Orb_p(M)$ in $M$,
such that near $p$ $M$ is defined by
\begin{equation}
\begin{aligned}
&\Im w=\varphi(z,\bar z,\Re w, \Re w'), \\
&\Im w'=0, \\
&w''=0,
\end{aligned}
\end{equation}
where $\varphi$ is a real valued real-analytic function with
$\varphi(z,0,s, s')\equiv0$.
Moreover,
the local CR orbit of the point
$(z,w,w',w'')=(0,0,s',0)$, for
$s' \in \R^q$, is given by 
\begin{equation}
\begin{aligned}
&\Im w=\varphi(z,\bar z,\Re w,s') , \\
& w'=s' , \\
& w''=0 .
\end{aligned}
\end{equation}
\end{thm}

A CR submanifold $M$ where $\Orb_p(M)$ is of maximal dimension
is Levi-flat if and only if $\Orb_p(M)$ is a complex manifold,
that is when $q=d$.  This definition is the same as in the introduction
and also includes complex manifolds.

Let $M$ be a CR submanifold.
A function $f \colon M \to \C$ such that $\bar{L} f = 0$ for
every $\bar{L} \in T^{0,1} M$ is called a \emph{CR function}.  For example,
a restriction to $M$
of a holomorphic function defined in a neighborhood of $M$ is a CR function.
On the other hand, if $M$ and $f$ are both real-analytic, then $f$
extends to holomorphic function defined on a neighborhood of $M$.


\section{Segre varieties}

Let $U \subset \C^N$ be a connected open set,
and write ${}^*U = \{ z : \bar{z} \in U \}$.
Let $H \subset U$ be defined
by $r(z,\bar{z}) = 0$ and suppose that $r$ can be complexified (by
complexifying its power series)
as a function $r(z,\bar{w})$ on $U \times {}^*U$.  Let $p \in H$.

\begin{defn}
We write 
\begin{equation}
\Sigma_p(U,r) := \{ z \in U : r(z,\bar{p}) = 0 \} .
\end{equation}
We call $\Sigma_p(U,r)$ the \emph{Segre variety} of $H$ at $p$ with
respect to $r$.
\end{defn}

We need a short lemma that is proved in \cite{burnsgong:flat}
that says that a germ of a real-analytic function is irreducible
if and only if its complexification is irreducible.

\begin{lemma}
\label{lemma:irredlemma}
If $\rho$ is an irreducible germ of a real-analytic function near $0$
in $\C^N$, and $H := \{ z : \rho(z,\bar{z}) = 0 \}$ has dimension 
$2N-1$, then for any neighborhood $U$ of 0, there is a smaller neighborhood
$U' \subset U$ of 0, such that if $\hat{\rho}$ is any real-analytic function
on $U$
that vanishes on an open set of $H^* \cap U'$, then $\rho$ divides
$\hat{\rho}$ on $U'$.  Further, $\rho$ is irreducible as a germ of a holomorphic
function near origin in $\C^{2N}$.
\end{lemma}

The variety $\Sigma_p(U,r)$ depends on both $U$ and $r$.  However, it
is possible to talk uniquely about a germ $\Sigma_p(H)$ not depending
on $U$ and $r$.
First, we note that the ideal $I_p(H)$
of germs at $p$ of real-analytic functions vanishing on $H$
is generated by some real-analytic germ $r$.  Let us take
a small enough connected
neighborhood $U$ of $p$ such that $r$ complexifies
to $U \times {}^*U$ and such that $\{ r(z,\bar{p}) = 0 \}$ contains
only components passing through $p$.  If $\varphi$ is another real-analytic
germ defining the ideal $I_p(H)$, then $\varphi = \alpha r$
where $\alpha(p,\bar{p}) \not= 0$.  It is then easy to see that as
germs at $p$
we have $\{ r(z,\bar{p}) = 0 \} = \{ \varphi(z,\bar{p}) = 0 \}$.
Therefore there is a well defined germ of a complex variety at $p$.
Denote by $(\Sigma_p(U,r) , p)$ the germ of 
$\Sigma_p(U,r)$ at $p$.

\begin{defn}
Define the germ
$\Sigma_p(H)$ as the germ $(\Sigma_p(U,r) , p)$
for $U$ small enough and $r$ as given above.
\end{defn}

That is, for each point, we can pick a small enough neighborhood
and a defining function $r$ such that $\Sigma_p(H)$ is well defined.
We have proved above that for any $U$ and $r$ we have, as germs at $p$,
\begin{equation}
\Sigma_p(H) \subset ( \Sigma_p(U,r) , p) .
\end{equation}

The following proposition is classical and not hard to prove by
complexification.

\begin{prop} \label{prop:cplxinsegre}
Let $U \subset \C^N$ be an open set,
$S \subset U$ be a real-analytic subvariety.
Suppose that
$r \colon U \to \R$ is real-analytic, complexifies to $U \times {}^*U$,
vanishes on $S$, and suppose $W \subset U$
is a complex subvariety such that $W \subset S$.  Then for
$p \in W$ we have
$W \subset \Sigma_p(U,r)$.
\end{prop}


\section{Degenerate singularities}

In the sequel,
we say $H \subset \C^N$
is a \emph{local real-hypervariety} to mean that
$H$ is a closed subvariety of some open set $U \subset \C^N$.
Also instead of writing $\overline{H^*} \cap U$ we simply use
$\overline{H^*}^{\mathit{rel}}$ to mean the relative closure of $H^*$
in $U$ (or equivalently the closure in the subspace topology on $H$).

\begin{defn}
Let $H \subset \C^N$ be a Levi-flat local real-hypervariety.
A point $p \in H$ is said to be a \emph{Segre-degenerate singularity}
if $\Sigma_p(H)$ is of dimension $N$, that is, $\Sigma_p(H) = (\C^N,p)$.
\end{defn}

In other words, $p$ is a degenerate singularity of $H$ if
$z \mapsto r(z,\bar{p})$ is identically zero for every local
defining function of $H$ at $p$.

Suppose that $(V,p) \subset (H,p)$ is a germ of a complex subvariety.
By Proposition~\ref{prop:cplxinsegre}
$(V,p) \subset \Sigma_p(H)$.
As a nonsingular Levi-flat hypersurface contains a unique
nonsingular complex analytic hypersurface through every point,
we obtain the following well-known result.

\begin{prop} \label{prop:nonsingsigma}
Let $H \subset \C^N$ be a Levi-flat real analytic manifold
of dimension $2N-1$.  Then $\Sigma_p(H) \subset (H,p)$
for every $p$ and $\Sigma_p(H)$ is nonsingular.
\end{prop}

The proposition implies that only singular points can be Segre-degenerate.
In fact the set of Segre-degenerate singularities must be small.

\begin{prop} \label{prop:degennm2}
Let $H \subset \C^N$ be a Levi-flat local real-hypervariety.
The set $S \subset H$ of Segre-degenerate singularities is contained
in a complex subvariety of (complex) dimension $N-2$ or less.
\end{prop}

\begin{proof}
Fix a point $p \in S$ and take a defining function $r$ for $H$
in some neighborhood $U$ of $p$.  Let us suppose that $r$ complexifies to
$U \times {}^*U$.  We can assume that $U$ is a polydisc.
As $r$ is real then $x \in \Sigma_y(U,r)$ implies
$y \in \Sigma_x(U,r)$.  Hence, $p \in \Sigma_q(U,r)$
for every $q \in U$.  Take the set
\begin{equation}
S_r = \bigcap_{q \in U} \Sigma_q(U,r) .
\end{equation}
As $\Sigma_q(U,r)$ must be proper subvariety for most $q$ (otherwise
$r$ would be identically zero),  $S_r$ is a proper complex subvariety
of $U$.  In fact we obtain that $S_r \subset H$ because if
$r(z,\bar{q})$ is zero for all $q$, then in particular 
$r(z,\bar{z}) = 0$.  Obviously we also have $S \subset S_r$.  We simply
need to show that $S_r$ must not be of dimension $N-1$.

Let us suppose that $S_r$ contains a branch $X$ of dimension $N-1$.
We assume that $p \in X$.
As $U$ is a polydisc we can choose a
a defining function $f$ for $X$ in $U$, such that
$f$ generates the ideal $I_U(X)$ of functions holomorphic on $U$
that vanish on $X$.
As we have that
\begin{equation}
f \mid r(\cdot,\bar{q}) \qquad \text{and} \qquad
\bar{f} \mid r(q,\bar{\cdot}) \qquad \text{for all $q \in U$},
\end{equation}
we get that $\abs{f(z)}^2$ divides $r(z,\bar{z})$.

We can assume that $r$ is not divisible by any $\abs{f(z)}^2$.  If
$H$ had a complex component we could replace the factor
$\abs{f(z)}^2$ by for example ${(\Re f(z))}^2 + {(\Im 2f(z))}^2$.
Therefore, $S_r$ must be of dimension $N-2$ or lower.
\end{proof}

The set of Segre-degenerate singularities is also closed.  In fact,
we have proved that
for every given defining function the set of Segre-degenerate singularities
with respect to that defining function must be a complex subvariety.

\begin{prop} \label{prop:defclosed}
Let $H \subset \C^N$ be a Levi-flat local real-hypervariety.
Then the set $S$ of Segre-degenerate singularities is closed.

In fact, when $r$ is a defining function for $H$ near $p$
that complexifies to $U \times {}^*U$ for some neighborhood
$U$ of $p$, then the set
\begin{equation}
S_r := \{ q \in U \colon \dim \Sigma_q(U,r) = N \}
\end{equation}
is a complex subvariety, and $S_r \subset H$.
\end{prop}

\begin{proof}
The proposition follows at once from the proof of
Proposition~\ref{prop:degennm2} once we notice that the
two definitions of the set $S_r$ agree.
\end{proof}

A useful corollary of this result is that if $p$ is
not a Segre-degenerate singularity then we can fix
a neighborhood $U$ of $p$ and a defining function $r$ such that
$H$ is not a Segre-degenerate singularity with respect to $r$ at
any point of $U$.


\section{Leaves at singular points}

We need the following well known result.  See
Diederich and Forn\ae ss~\cite{DF:realbnd} (the claim in section 6).

\begin{lemma}[Diederich-Forn\ae ss] \label{DF:lemma}
Let $S \subset \C^N$ be a local real-analytic subvariety.  For every
$p \in S$, there exists a neighborhood $U$ of $p$ such that
for every $q \in U$ and every germ of a complex variety
$(V,q) \subset (S,q)$, there exists a (closed) complex subvariety $W \subset
U$ such that $(V,q) \subset (W,q)$ and such that $W \subset S \cap U$.
\end{lemma}

This lemma has an interesting and useful 
corollary that was pointed out to the author by Xianghong Gong.

\begin{cor} \label{cor:localcplxisglobalcplx}
If $X \subset \C^N$ is a local real-analytic subvariety such that
for every $p \in X_{reg}$ there exists a neighborhood $U$ of
$p$ such that $X \cap U$ is a complex manifold.  Then $X$ is a local
complex analytic subvariety.
\end{cor}

By $X_{reg}$ we mean the set of points near
which $X$ is a real-analytic manifold (of any dimension).

\begin{proof}
Take $q \in X$ be a singular point.  By considering
the local complex subvariety $X_{reg}$ and appealing to
Lemma~\ref{DF:lemma}, there exists a small
neighborhood $U$ of $q$ and a complex subvariety $X' \subset U$
such that $X_{reg} \cap U \subset X' \subset X \cap U$.
As $\overline{X_{reg}}^{\mathit{rel}} = X$ we are done.
\end{proof}

We also need the following lemma of Forn\ae ss.
The proof is
given in \cite{kohn:subell}, Theorem 6.23.
The statement we need
is stronger, though more technical,
and follows from minor modification of the proof in \cite{kohn:subell}.
We reproduce the proof here with the
necessary modifications.

\begin{lemma}[Forn\ae ss] \label{F:lemma}
Let $S \subset \C^N$ be a local real-analytic subvariety.
Suppose that $W_k \subset S$ is a sequence of
local complex subvarieties with $\dim W_k \geq m$.
If $p \in S$ is a cluster point of this sequence, then
there exists a neighborhood $U$ of $p$,
a subsequence $\{ W_{k_j} \}$,
with $p$ still as a cluster point,
and
a complex subvariety $W \subset S \cap U$
with $\dim W \geq m$, such that
$W$ contains the set $C$ of the cluster points (in $U$) of 
$\{ W_{k_j} \cap U \}$.  Furthermore, no such subvariety $W$ of 
dimension less than $m$ exists.
\end{lemma}

\begin{proof}
Let $U$ be a neighborhood
of $p$ such that all $W_k \cap U$ extend to a closed
complex subvariety of $U$ of dimension at least $m$, 
so assume that $W_k$ are closed subvarieties of $U$.
We can also assume that the defining equation 
$r(z,\bar{z})$ complexifies to $U \times {}^*U$.

Let $p^{(1)}$ be a cluster point of $\{ W_k \}$.
We pass to a subsequence to find $p_k^{(1)} \in W_k$ such that
$\lim p_k^{(1)} = p^{(1)}$.  We proceed inductively.
Let $C_n$ be the set of cluster points (in $U$)
of the sequence $\{ W_k \}$
at the $n$th step.  Let $d$ be the supremum of
the distance of a point
$q \in C_n$ to the set $P_n = \{ p^{(1)} , \ldots , p^{(n-1)} \}$.
We choose $p^{(n)}$ to be the point of $C_n$ that is of distance
at least $\frac{n}{n+1} d$ from $P_n$.  We again pass to
a subsequence of $\{ W_k \}$ and choose $p_k^{(n)} \in W_k$
such that $\lim p_k^{(n)} = p^{(n)}$.
Using diagonalization we obtain a subsequence $\{ W_k \}$ such
for each $j$ we have $p_k^{(j)} \in W_k$ and
$\lim p_k^{(j)} = p^{(j)}$.  The set $\{ p^{(j)} \}$ is dense
in the set $C$ of limit points of $\{ W_k \}$.
As $p_k^{(n)}, p_k^{(j)} \in W_k$ we have 
that $r(p_k^{(n)}, \bar{p}_k^{(j)}) = 0$ by
Proposition~\ref{prop:cplxinsegre}.  Taking limits and
using the density of $\{ p^{(j)} \}$ in $C$, we have
that $r(z,\bar{w}) = 0$ for all $z,w \in C$.
Define closed complex subvarieties $W', W \subset U$ by
\begin{equation}
W' := \bigcap_{q \in C} \Sigma_q(U,r) 
\qquad \text{and} \qquad
W := \bigcap_{q \in W'} \Sigma_q(U,r) .
\end{equation}
If $q \in W'$ and $c \in C$, then $r(q,\bar{c}) = 0$ and hence
by reality of $r$, $r(c,\bar{q}) = 0$.  Therefore
$C \subset W \subset W'$.  Furthermore $r(z,\bar{z}) = 0$
for all $z \in W$ and so $W \subset S$.

Let us show that $W$ must be of (complex) dimension at least $m$.
Suppose that $W$ is of dimension $d$.
Pick a point
$q \in C \cap W_{reg}$.  If such a point does not exist, we
then we have $C \subset W_s$ and we could have taken $W_s$ instead of $W$.
We can assume that in 
a small neighborhood of $q$ we have local holomorphic coordinates
such that $q$ is the origin and
$W$ is given by $z_{d+1} = \cdots z_N = 0$.  We can assume that
$W$ and $W_k$ are closed in a neighborhood of the closure of the
unit polydisc $\overline{\Delta}$.  We have that
for large $k$ we have that if $z \in W_k \cap \Delta$, then
$\abs{z_j} < \nicefrac{1}{2}$ for $j=d+1,\ldots,N$.
Therefore, the projection of
$W_k \cap \Delta$ onto $W \cap \Delta$ must be proper.  Hence $d \geq m$.
\end{proof}

We require the following result.
A somewhat weaker version of this fact
was proved in \cite{burnsgong:flat},
in the case the point is not a Segre-degenerate singularity,
and not concluding that 
$W \subset \overline{H^*}^{\mathit{rel}}$.  The conclusion
that $W \subset \overline{H^*}^{\mathit{rel}}$ was also proved for a special
case in \cite{Lebl:projlf}.

\begin{lemma} \label{dimnleaf:lemma}
Let $H \subset \C^N$ be a Levi-flat local real-hypervariety.
Suppose that $p \in \overline{H^*}^{\mathit{rel}}$, then there exists a
neighborhood $U$ of $p$ and an irreducible
complex subvariety $W \subset U$ of dimension $N-1$
such that $W \subset \overline{H^*}^{\mathit{rel}}$ and $p \in W$.
\end{lemma}

\begin{proof}
Let $U$ be a neighborhood of $p$ as in Lemma~\ref{DF:lemma}.

We take a sequence $q_k \to p$,
$q_k \in H^*$.  For each $q_k$ we apply Proposition~\ref{prop:nonsingsigma}
and Lemma~\ref{DF:lemma} to find a complex subvariety $W_k \subset U$
of dimension $N-1$ such that $q_k \in W_k$ and $W_k \subset H$.  We
can also assume that $W_k \subset \overline{H^*}^{\mathit{rel}}$.  That is because
there can be at most finitely many $W_k$ such that
$W_k \not\subset \overline{H^*}^{\mathit{rel}}$.  By Lemma~\ref{F:lemma}
we find a subsequence (calling it again $\{ W_k \}$) and
a complex subvariety $W \subset H$ of dimension $N-1$ that contains
all the cluster points of $\{ W_k \}$ in $U$.

If $H$ has no branch of dimension $2N-2$, which is a complex
variety at some point, then we are done.

The set
$C$ of cluster points of $\{ W_k \}$
is a subset of $\overline{H^*}^{\mathit{rel}}$.
By Lemma~\ref{F:lemma}, the set
$C$ of cluster points of $\{ W_k \}$
cannot be contained in the set $S$ of Segre-degenerate
singularities of $H$,
because $S$ would be contained in a complex subvariety of
dimension $N-2$ or less.

Let us move to a point $q \in C \setminus S$.
By Proposition~\ref{prop:cplxinsegre}
all germs of complex subvarieties of dimension $N-1$ through $q$
contained in $H$ must be
subsets (and hence branches) of the
Segre variety $\Sigma_q(U',r)$ for some neighborhood $U'$ of $q$,
which is a proper subvariety.
After perhaps a linear change of variables
we can assume that $U'$ is small enough such that
we can apply Weierstrass preparation theorem on $r$ with respect
the $z_N$ variable to obtain a new defining function $\tilde{r}$
\begin{equation}
\tilde{r}(z,\bar{z})
=
z_N^d + \sum_{j=0}^{d-1} p_j(z',\bar{z}',\bar{z}_N) z_N^j ,
\end{equation}
where we use the we use the notation $z = (z_1,\ldots,z_N) = (z',z_N)$.
We know that $W_k \cap U'$ are contained in $H$ and therefore
for a sequence $q^{(k)} \in W_k$
\begin{equation}
W_k \cap U' \subset
\Bigl\{ z \in U' : z_N^d + \sum_{j=0}^{d-1}
p_j\bigl(z',{\bar{q}^{(k)\prime}},\bar{q}^{(k)}_N\bigr) z_N^j
\Bigr\} = \Sigma_{q^{(k)}} (U',\tilde{r}) .
\end{equation}
That means that $W_k \cap U'$ is multigraph of the holomorphic
function $f_k \colon V' \to \C_{sym}^d$ for some neighborhood
$V' \subset \C^{N-1}$.
Here $\C_{sym}^d$ is the $d$th symmetric power and the multigraph
is the set $\{ (z,w) \colon w \in f_k(z) \}$.
For more information on complex varieties as
multigraphs of holomorphic mappings into the symmetric spaces
see \cite{Whitney:book}.

The functions $f_k$ are bounded and hence there exists a convergent
subsequence, these converge to some 
$f \colon V' \to \C_{sym}^d$.  Let us call $W'$
the multigraph of $f$.
As $W_k \cap U' \subset \overline{H^*}^{\mathit{rel}} \cap
U'$ then $W'$ is contained
in $\overline{H^*}^{\mathit{rel}} \cap U'$.  In fact the set of cluster points of $W_k \cap
U'$ is in fact $W'$ so $W' \supset C \cap U'$.
If $C \setminus W'$ is nonempty, we could repeat the procedure
to get another branch.  We only need to repeat the procedure
finitely many times as $W'$ is of dimension $N-1$ and therefore
$C$ cannot be contained a complex subvariety of larger dimension.
Therefore we can assume that $W' = C \cap U'$.

Hence $C \setminus S$ is a closed complex subvariety of
$U \setminus S$.  As $S$ is a subset of a complex variety
of dimension $N-2$, we can use the Remmert-Stein theorem to extend
$C$ to a closed convex subvariety $W'' = \overline{C \setminus
	S}^{\mathit{rel}} \subset
\overline{H^*}^{\mathit{rel}}$ of dimension $N-1$.  It is not hard
to see that $C = W''$ because $C$ is closed and all $W_k$ were subsets
of $\overline{H^*}^{\mathit{rel}}$.
\end{proof}

As we said in the introduction, the lemma gives an alternative
characterization of singular Levi-flat real-hypervarieties.  That is
a real-hypervariety
$H \subset U$ is Levi-flat if and only if all
the components of $U \setminus \overline{H^*}^{\mathit{rel}}$ are pseudoconvex.

The following corollary of Lemma~\ref{dimnleaf:lemma} was already proved
in \cite{burnsgong:flat} in the case that $H$ is not Segre-degenerate.

\begin{cor}
Suppose that $H \subset \C^N$ is a Levi-flat local real-hypervariety
that is reducible as a germ at $p \in \overline{H^*}^{\mathit{rel}}$ into two
distinct
germs of real-hypervarieties $(H_1,p)$ and $(H_2,p)$.  Then $H_s$ is of dimension
at least $2N-4$ and there exists a local complex subvariety $X$
of (complex) dimension $N-2$ such that
$X \subset (\overline{H^*}^{\mathit{rel}})_s \subset H_s$.
\end{cor}

\begin{proof}
We can assume that we are working in a neighborhood $U$ of $p$
such that $H_1$, $H_2$, and $H$ are closed subvarieties of $U$.
Let us assume that $H_1$ is irreducible as a germ at $p$ and
assume that $H_2$ does not have $H_1$ as one of its branches.  We can
thus assume that $H_1^*$ and $H_2^*$ do not meet on a set of
dimension $2N-1$.
There must exist two irreducible complex hypervarieties
$W_1 \subset \overline{H_1^*}^{\mathit{rel}}$ and
$W_2 \subset \overline{H_2^*}^{\mathit{rel}}$ through $p$ by Lemma~\ref{dimnleaf:lemma}.
As both $W_1$ and $W_2$ are in the closure of $H_1^*$ and $H_2^*$
and since $H_1^* \cap H_2^*$ is not of dimension $2N-1$
it must be that $W_1 \cap W_2$ lies in the singularity
$(\overline{H^*}^{\mathit{rel}})_s$.
As $W_1 \cap W_2$ must be of (complex) dimension at least $N-2$, the
corollary follows.
\end{proof}


\section{Leaf-degenerate points} \label{section:ldeg}

\begin{defn}
Let $H \subset \C^N$ be a Levi-flat local real-hypervariety.
For $p \in H$,
Lemma~\ref{DF:lemma} implies that there exists
a neighborhood $U$ of $p$ such that each germ of a complex
subvariety $(V,p) \subset (H,p)$ extends to a (closed) subvariety of $U$.
Hence define
$\Sigma'_p(H)$ as the germ at $p$ of the union of complex
subvarieties $V$ of $U$ of (complex) dimension $N-1$ such that
$V \subset H \cap U$.

If $p \in H$ is such that $\Sigma'_p(H)$ is not a 
complex variety of dimension $N-1$ then we say that $p$
is a \emph{leaf-degenerate point}.
\end{defn}

We show that the above definition of leaf-degenerate points
is equivalent to the definition from the introduction.

\begin{lemma} \label{lemma:sprimeprop}
Let $H \subset \C^N$ be a Levi-flat local real-hypervariety.
If $p \in \overline{H^*}^{\mathit{rel}}$ then $\Sigma_p'(H)$ is nonempty
and in fact contains a local complex hypervariety $W$
such that $p \in W \subset \overline{H^*}^{\mathit{rel}}$.

Furthermore, if $p \in \overline{H^*}^{\mathit{rel}}$ is a leaf-degenerate point
then $p$ is a Segre-degenerate singularity, and
$\dim (\overline{H^*}^{\mathit{rel}})_s \geq 2N-4$.
\end{lemma}

\begin{proof}
Lemma~\ref{dimnleaf:lemma} says that $\Sigma_p'(H)$ is nonempty
and contains a complex hypervariety $W \subset \overline{H^*}^{\mathit{rel}}$.
By Proposition~\ref{prop:cplxinsegre} we have
$\Sigma_p'(H) \subset \Sigma_p(H)$.  Thus if $p$ is leaf-degenerate,
$\Sigma_p'(H)$ must contain infinitely many distinct complex
subvarieties of dimension $N-1$, and therefore $\Sigma_p(H)$ must be
open.

As $\Sigma_p'(H)$ is a union of infinitely many germs of
complex subvarieties of dimension $N-1$, suppose that
$V_1$ and $V_2$ are two such subvarieties with no component in common.
As there are infinitely many such subvarieties in $\Sigma'_p(H)$,
and only finitely many complex subvarieties can contain
points of $H \setminus \overline{H^*}^{\mathit{rel}}$ ($H$ can have at most finitely
many components through $p$ of dimension less than $2N-1$),
we can assume that $V_1$ and $V_2$ are subsets of
$\overline{H^*}^{\mathit{rel}}$.
Then $V_1 \cap V_2$ is a complex variety of dimension $N-2$,
and we know that $V_1 \cap V_2 \subset (\overline{H^*}^{\mathit{rel}})_s$ since at
nonsingular
points of $\overline{H^*}^{\mathit{rel}}$ we have a unique
leaf.
\end{proof}

We can now classify those singular sets which are completely degenerate.
That is, those singular sets where we cannot move to a generic point and
expect a leaf-nondegenerate points in a neighborhood.

\begin{lemma} \label{lemma:degensingarevar}
Let $H \subset \C^N$ be a Levi-flat local real-hypervariety,
such that $E = (\overline{H^*}^{\mathit{rel}})_s$
is a connected real-analytic submanifold.
Suppose that the set $S$ of leaf-degenerate points is
dense in $E$, then $E$ must be a complex submanifold of
dimension $N-2$.
\end{lemma}

\begin{proof}
Let $p \in S \subset E$.  The set $S$ is a subset of the Segre-degenerate
singularities (see Lemma~\ref{lemma:sprimeprop}), and the Segre-degenerate singularities must be contained
in a complex subvariety of (complex) dimension $N-2$ or less
(see Proposition~\ref{prop:degennm2}).  As $S$
is dense in $E$, then $E$ must be of (real) dimension $2N-4$ or less.

As in the proof of Lemma~\ref{lemma:sprimeprop},
we have two complex subvarieties $V_1$ and $V_2$
of dimension $N-1$ contained in $\overline{H^*}^{\mathit{rel}}$
with no branch in common.
As $V_1 \cap V_2 \subset E$ is a complex subvariety of dimension $N-2$
and $E$ is a connected real-analytic submanifold
of dimension at most $2N-4$,
we have $V_1 \cap V_2 = E$.
\end{proof}


\section{Generic singular set}

In \cite{Lebl:lfnm} the author proved the following theorem.

\begin{thm} \label{lemma:lfnmlemma}
Let $H \subset \C^N$ be a Levi-flat local real-hypervariety.
Let $M \subset \overline{H^*}^{\mathit{rel}}$ be a real-analytic generic submanifold, then
$M$ is not a minimal CR submanifold.
\end{thm}

In the present paper we extend the proof of this
result to prove the following lemma.

\begin{lemma} \label{lemma:genericsing}
Let $H \subset \C^N$ be a Levi-flat local real-hypervariety.

Suppose that $E = (\overline{H^*}^{\mathit{rel}})_s$ is
a connected generic real-analytic submanifold.
Then $E$ is a generic Levi-flat submanifold
of dimension $2N-2$.
\end{lemma}

Large parts of the following proof already appeared in \cite{Lebl:lfnm}
in the proof of Theorem~\ref{lemma:lfnmlemma}.  As we need to modify
the proof in many places, we simply reproduce the entire proof here
with modifications as needed.  Some of the techniques used are similar
to those of Burns and Gong~\cite{burnsgong:flat}.
First we need the following short lemma,
which also appears in \cite{Lebl:lfnm}.  We need a somewhat
stronger conclusion than what is stated in \cite{Lebl:lfnm}
and hence we reprove it here.

\begin{lemma}
\label{lemma:interleviflat}
Let $H_1, H_2 \subset \C^N$, $N \geq 2$, be
two connected nonsingular real-analytic Levi-flat
hypersurfaces.  If $p \in H_1 \cap H_2$,
then there exists a neighborhood $U$ of $p$
and a complex subvariety $A \subset U$ such that
$(U \cap H_1 \cap H_2) \setminus A$ is a generic Levi-flat
submanifold of dimension $2N-2$.

In fact, if $M = H_1 \cap H_2$ is a connected
real-analytic CR submanifold,
then $M$ is either a complex hypersurface or a generic
Levi-flat submanifold of dimension $2N-2$.
\end{lemma}

\begin{proof}
Take $U$ to be a small enough neighborhood of $p$
such that $H_1$ and $H_2$ are closed subsets.
Change coordinates such that
$p = 0$, and in $U$, $H_1$ is given by $\Im z_1 = 0$, and $H_2$ is given
by $\Im f = 0$ for a holomorphic function with nonvanishing differential.
Define $A$ to be the complex subvariety of $U$
where the differentials $dz_1$
and $df$ linearly dependent.  Outside of $A$ we can change coordinates
once again and assume that $H_2$ is given by $\Im z_2 = 0$ hence the
intersection is generic Levi-flat of dimension $2N-2$.  If the
differentials are everywhere dependent, then $f$ depends only on $z_1$
and in this case the intersection is a complex hypersurface.  The first
part of the lemma is proved.

Thus assume that $M = H_1 \cap H_2$ is a connected real-analytic CR
submanifold.
At $p \in M$ there exist the complex hypersurfaces $W_1 \subset H_1$
and $W_2 \subset H_2$ (closed in $U$).  We note that $W_1 \cap W_2 \subset
H_1 \cap H_2 = M$.
If $W_1 = W_2$ then $M$
is a complex hypersurface and we are done.  Otherwise $W_1 \cap W_2$
is of (complex) dimension $N-2$.  As above assume that $H_1$ is
$\{ \Im z_1 = 0 \}$ and $H_2$ is $\{ \Im f = 0 \}$, where $f(p) = 0$.
Unless $M = H_1 = H_2$
we can assume that $\Im f$ is positive somewhere on $\{\Im z_1 = 0\}$,
and without loss of generality it can be on $\{z_1 = 0\}$.  Unless $M$ is
of dimension $2N-2$, this would mean that $\Im f \geq 0$ on
$\{ \Im z_1 = 0 \}$
and in fact $\Im f(z) > 0$ for some $z$ on $\{ z_1 = 0 \}$.
By the maximum principle this is impossible as $f(p) = 0$.
\end{proof}

To be able to assume that $H$ is irreducible we need the following
proposition.

\begin{prop} \label{prop:tech}
Let $H \subset \C^N$ be a Levi-flat local real-hypervariety.  Also
suppose that $(\overline{H^*}^{\mathit{rel}})_s$
is a connected real-analytic CR submanifold
that is not a generic Levi-flat submanifold of dimension $2N-2$
nor a complex submanifold of (complex) dimension $N-1$.

Then there exists a neighborhood $U$ of $p$, and real-hypervariety
$\tilde{H} \subset H \cap U$, irreducible as germ at $p$,
such that
$(\overline{H^*}^{\mathit{rel}})_s \cap U =
(\overline{\tilde{H}^*}^{\mathit{rel}})_s$.
\end{prop}

\begin{proof}
Pick a point $p \in (\overline{H^*}^{\mathit{rel}})_s$.
Take the irreducible components $H_1,\ldots,H_k$ of $H$ at $p$.  These
are irreducible real-analytic subvarieties of some neighborhood $U$ of
$p$.  We can also assume that $U$ is such that
$(\overline{H^*}^{\mathit{rel}})_s \cap U$ is connected.
As there are only finitely many components $H_j$,
and $(\overline{H_j^*}^{\mathit{rel}})_s \subset
(\overline{H^*}^{\mathit{rel}})_s \cap U$,
then the manifold $(\overline{H^*}^{\mathit{rel}})_s \cap U$ is either
the singularity of some $(\overline{H_j^*}^{\mathit{rel}})_s$ or
there must exist a point $q \in (\overline{H^*}^{\mathit{rel}})_s \cap U$
where $\overline{H^*}^{\mathit{rel}}$ is a union of at least two
real-analytic submanifolds of dimension $2N-1$.  Applying
Lemma~\ref{lemma:interleviflat} would violate the hypothesis.
\end{proof}

Now we have the tools to prove Lemma~\ref{lemma:genericsing}.

\begin{proof}[Proof of Lemma~\ref{lemma:genericsing}]
We can move to a generic point on $E =
(\overline{H^*}^{\mathit{rel}})_s$.
Therefore, we can avoid arbitrary
proper complex local subvarieties, as $E$
is not contained in any such subvariety
($E$ is a generic submanifold).  Therefore, we can assume that $H$ does not
have a Segre-degenerate singularity at $p \in E$ by applying
Proposition~\ref{prop:degennm2}.

By Proposition~\ref{prop:tech}
we can assume that
$H$ is irreducible as a germ
at $p$.

We fix a connected neighborhood $U$ of $p$,
and a defining equation $r(z,\bar{z}) = 0$ for $H$ such that
$r$ complexifies to $U \times {}^*U$.  We 
define all Segre varieties using this $U$ and $r$ from now on.  We
also assume that both $H$ and $E$ are closed subsets of $U$.
We can assume that $H$ is irreducible in $U$, and $r$ is also irreducible
as a holomorphic function of $z$ and $\bar{z}$, see
Lemma~\ref{lemma:irredlemma}.

We can assume that $U$ is small enough to be able to
apply Lemma~\ref{DF:lemma}.  Thus we 
write $\Sigma'_q(H;U)$ when we are talking
about the smallest
(closed) complex subvariety of $U$ contained in $H$ and
containing $\Sigma'_q(H)$.

By Proposition~\ref{prop:defclosed} we know we could have picked
$U$ small enough such that $\dim \Sigma_q(H) = N-1$
for all $q \in U$.

By Lemma~\ref{dimnleaf:lemma},
$\Sigma_p'(H;U)$ is nonempty.
As $E$ is generic, no branch of $\Sigma'_p(H;U)$ contains $E$.
Furthermore because $E$ is generic,
no branch of $\Sigma_p'(H;U)$ lies in $E$.
Thus there must exist a $q$ on $E$ such that
$\Sigma'_q(H;U)$ intersects $H^*$.  We set $p=q$ and again apply
Proposition~\ref{prop:tech} to assume that $H$ is irreducible
at $p$.

We find a point $\zeta \in H^* \cap \Sigma_p'(H;U)$.
As there is a unique complex
hypersurface in $H$ through $\zeta$, we know that $\Sigma_{\zeta}'(H;U)$
contains a branch of $\Sigma_p'(H;U)$.

We pick $\zeta$ to lie in a topological component of
$(\Sigma_p'(H;U))_{reg} \cap H^*$
(where $(\Sigma_p'(H;U))_{reg}$ is the nonsingular part of
$\Sigma_p'(H;U)$),
such that $p$ is in the closure of this component.
Now pick a nonsingular
real-analytic curve $\gamma \colon (-\epsilon,\epsilon) \to H$
such that $\gamma(0) = \zeta$, $\{\gamma\} \subset H^*$, and such that
$\gamma$ is transverse to the Levi-foliation of $H^*$.
The function
$t \mapsto r(p,\bar{\gamma}(t))$ is not identically zero.  If it were
identically zero,
then $\Sigma_p(U,r)$ would contain an open set
(the union of representatives of $\Sigma_{\gamma(t)}'(H)$)
and we assumed that $H$ was Segre-nondegenerate at $p$ with respect to $r$.

We complexify $t$ in
$r(z,\bar{\gamma}(t))$, and
apply the Weierstrass preparation theorem to
$r(z,\bar{\gamma}(t))$ in some neighborhood $U' \times D$ where
$p \in U' \subset U$ and $D \subset \C$.  We obtain
\begin{equation}
F(z,t) = t^m + \sum_{j=0}^{m-1}a_j(z)t^j ,
\end{equation}
with the same zero set in $U' \times D$ as
$r(z,\bar{\gamma}(t))$.  Let $\Delta \subset U'$
be the discriminant set of $F$.  Then near each point
of $U' \setminus \Delta$ we (locally) have 
$m$ holomorphic functions $\{e_j\}_1^m$ that
are solutions of
$F(z,e_j(z))=0$.  We wish to study the set where at
least one of the $e_j$ is real-valued, that is
$e_j-\bar{e_j} = 0$.  We define
\begin{equation} \label{eq:phidef}
\varphi(z,\bar{z}) =
i^m \prod_{j,k=1}^m \bigl(e_j(z)-\overline{e_k(z)}\bigr).
\end{equation}
The expression on the right is real-valued and
symmetric both in the $e_j(z)$ and the $\overline{e_k(z)}$.
Therefore, after complexification we have
a well defined function on
$(U' \times {}^*{U'}) \setminus (\Delta \times {}^*{\Delta})$,
which extends to be continuous
in all of $U' \times {}^*{U'}$
and thus holomorphic in $U' \times {}^*{U'}$,
see \cite{Whitney:book}.
Thus we have a real-analytic function
$\varphi \colon U' \to \R$
that is locally outside of $\Delta$ given by 
\eqref{eq:phidef}.

Let $K = \{ z \in U' : \varphi(z,\bar{z}) = 0 \}$.
We have that
$\Sigma_{\gamma(0)}'(H;U) \cap U'$ is a subset of $K$.
We cannot immediately conclude that $\Sigma_{\gamma(t)}'(H;U) \cap U'$
is a subset of $K$ for $t$ other than zero as
$\{ \gamma \}$ might not lie in $U'$.

We pick a point
\begin{equation}
\zeta' \in \bigl(\Sigma_{\gamma(0)}'(H;U)\bigr)_{reg} \cap H^* \cap U'
\end{equation}
As $\zeta$ was in the topological component of
$\bigl(\Sigma_{\gamma(0)}'(H;U)\bigr)_{reg} \cap H^*$ containing
$p$ in its closure,
we pick a path from $\zeta'$
to $\zeta$ in $\Sigma_{\gamma(0)}'(H;U) \cap H^*$ and
a finite sequence of overlapping neighborhoods $\{V_j\}$
whose union contains the path and such that
inside each $V_j$,
$H$ is given by $\Im f_j(z) =0$ (for
some $f_j$ holomorphic in $V_j$).
We assume that $\zeta' \in V_0 \subset U'$.
The Levi-foliation is given by $f_j (z) = c$ for real
$c$, and these sets must agree on $V_j \cap V_k$.  That is, we
have a nonsingular holomorphic codimension one foliation of a neighborhood
of the path from $\zeta'$ to $\zeta$.
Therefore for some small interval of $t$, we have that the sets
$\Sigma_{\gamma(t)}'(H;U) \cap V_0$ are nonempty and
are in fact equal to sets $\{ z : f_0(z) = c(t) \}$ for some real $c(t)$.

Thus $\Sigma_{\gamma(t)}'(H;U) \cap U'$ are subsets of $K$.
Therefore an open set of $H$ is a subset of $K$.  As $H$ is irreducible,
then $H \subset K$.

As $E$ is generic, 
$(E \cap U') \setminus \Delta$ is an open dense subset of
$E \cap U'$.  Hence at a point $q \in (E \cap U') \setminus \Delta$
there is a small neighborhood $U''$ of $q$ such that in $U''$
$\varphi$ is given by
$i^m \prod_{j,k=1}^m \bigl(e_j(z)-\overline{e_k(z)}\bigr)$.
As
$e_j(z)-\overline{e_k(z)}$ is pluriharmonic
its real and imaginary parts are pluriharmonic, meaning
that we can represent them as the imaginary part of a 
holomorphic function, that is
$e_j(z)-\overline{e_k(z)} = \Im f_{jk}(z) + i \Im g_{jk}(z)$.
Therefore $H \cap U''$
is contained in the zero set of
\begin{equation}
i^m
\prod_{j,k=1}^m (\Im f_{jk}(z) + i \Im g_{jk}(z)) .
\end{equation}
The zero set of each
$\Im f_{jk}(z) + i \Im g_{jk}(z)$ is a real-analytic subvariety
of real dimension $2N-1$ or $2N-2$.
Hence, there is some finite set of holomorphic functions $\{ h_k \}$
defined in $U''$ such that
\begin{equation}
	\overline{H^*}^{\mathit{rel}} \cap U'' \subset \{ z : \prod_k \Im h_k(z) = 0 \} .
\end{equation}
The set where the differentials of $h_k$ vanish is a complex subvariety
of $U''$.  As $E$ is generic, there must be a point $q' \in E$ and
a neighborhood $U'''$ of $q'$ such that
$\overline{H^*}^{\mathit{rel}} \cap U'''$
is contained in the union of finitely many nonsingular real-analytic
Levi-flat hypersurfaces.  Therefore
$\overline{H^*}^{\mathit{rel}} \cap U'''$ itself must be the union of finitely
many nonsingular real-analytic Levi-flat hypersurfaces.
We apply Lemma~\ref{lemma:interleviflat}.
Outside of a complex analytic subvariety $A$ of $U'''$ we have
that $E$ is a dimension $2N-2$ Levi-flat submanifold.  Again
as $E$ is generic, $(E \cap U''') \setminus A$ is nonempty.
Therefore there exists a point on $E$ where $E$ is Levi-flat
dimension $2N-2$ generic submanifold.  As $E$ is a
connected generic real-analytic submanifold,
then $E$ is a Levi-flat dimension $2N-2$ generic submanifold at
every point.
\end{proof}


\section{Intersections of Levi-flats with complex manifolds}

We need to see what happens to a Levi-flat
real-hypervariety when we intersect it with a complex manifold.  The following
lemma is useful in proving results about
Levi-flat hypervarieties by induction on dimension.

\begin{lemma} \label{lemma:interlemma}
Let $H \subset \C^N$ be a Levi-flat local real-hypervariety and
let $V \subset \C^N$ be a connected complex submanifold of positive
dimension $k$.  Suppose that there exists a point
$p \in \overline{H^*}^{\mathit{rel}} \cap V$.
Then exactly one of the following statements is true.
\begin{enumerate}[(i)]
\item $H \cap V$ is a complex variety of dimension $k-1$ and
$p$ is a leaf-degenerate point of $H$.
\item $H \cap V$ is a real-hypervariety of $V$ (is of dimension
$2k-1$).
\item $V \subset H$.
\end{enumerate}
\end{lemma}

\begin{proof}
By induction on codimension of $V$ it is enough to consider
$V$ of dimension $N-1$.
So let us suppose that $H \cap V$ is a proper subset of $V$
and hence of dimension
at most $2N-3$.
By Lemma~\ref{dimnleaf:lemma},
$\Sigma'_p(H)$ is nonempty, and contains at least 
one irreducible complex subvariety $W$ of dimension $N-1$.  If $W$
is not equal to $V$ then
$W \cap V$ must be of dimension $N-2$.  Hence $H \cap V$
must be of dimension at least $2N-4$.  We simply need to show that
if $H \cap V$ is a complex variety of complex dimension $N-2$ then
$p$ is a leaf-degenerate point.

By restricting to the correct 2 complex dimensional subspace it is
enough to consider $N=2$ with coordinates $(z,w) \in \C^2$
and it is also enough to consider
$V = \{ z = 0 \}$ and $p = 0$.
Suppose for contradiction that $V \cap H = \{ 0 \}$.
Take $V_\epsilon = \{ z = \epsilon \}$, for a small
complex $\epsilon$.  We note that
$V_\epsilon \cap H$ must be compact for small $\epsilon$
as $H$ is a closed subvariety of a neighborhood of the origin.
If $V_\epsilon \cap H$ was
isolated points (dimension 0) for all small $\epsilon$,
then dimension of $H$ would be $2$ which would
be a contradiction.  Thus $V_\epsilon \cap H$ must be of dimension 1
for $\epsilon$ arbitrarily close to 0 ($V_\epsilon \cap H$ cannot be
dimension $2$ and still compact as then $V_\epsilon$ would be
a subset of $H$).  Furthermore for $\epsilon$ arbitrary
close to zero we must have that $V_\epsilon \cap H^*$ is of dimension 1.
Since $V_\epsilon \cap H^*$ is of dimension 1
and $V_\epsilon \cap H$ is compact,
there must be infinitely many distinct leaves of $H^*$ that intersect
$V_\epsilon \cap H^*$.  As 
$V_\epsilon \cap H^*$ approaches the origin as $\epsilon$ goes to 0,
we see that infinitely many distinct leaves of $H^*$ must have the origin
in their closure.  By Lemma~\ref{DF:lemma} all of those leaves extend to
a subvariety of a neighborhood of the origin and the origin must be a
leaf-degenerate point.
\end{proof}

All three cases are possible.  The last two are obvious.  For
the first case consider the
Levi-flat hypersurface $H$ given by $\abs{z}^2-\abs{w}^2 = 0$.  Then
the set $V = \{ z = 0 \}$ intersects $H$ at the origin only.  The
origin is a leaf-degenerate point where for each $\theta$
we obtain a leaf $\{ z = e^{i\theta} w \}$.


\section{CR orbits of manifolds in Levi-flats}

We know that a complex subvariety of $H$ must lie in $\Sigma'_p(H)$,
however it is also true that a minimal CR submanifold that
lies inside $H$ also lies inside $\Sigma'_p(H)$ as we can
prove that its intrinsic complexification does.  In particular we prove
the following lemma.

\begin{lemma} \label{lemma:orbinsigma}
Let $H \subset \C^N$ be a Levi-flat local real-hypervariety
without degenerate singularities.
Suppose that $M \subset \overline{H^*}^{\mathit{rel}}$
is a connected real-analytic CR submanifold and $p \in M$ is a point
such that $\Orb_p(M)$ is of maximal dimension.

Then $\Orb_p(M) \subset \Sigma'_p(H)$.
\end{lemma}

\begin{proof}
First suppose that $M$ is minimal, that is, $\Orb_p(M) = (M,p)$ as germs.
If $M \not\subset (\overline{H^*}^{\mathit{rel}})_s$, then we can 
find a point $q \in M$ near which $\overline{H^*}^{\mathit{rel}}$ is nonsingular.
Thus suppose that $H$ is nonsingular.  In particular
we have
$H = \{ \Im f = 0 \}$ as germs at $q$
for a holomorphic function $f$ defined near $q$.  
As $M$ is minimal, then $f$ is constant on $M$ (the sets
$\{ f = c \} \cap M$ define CR submanifolds of same CR dimension as $M$).
Thus
$f$ is constant on the intrinsic complexification of $M$.  Therefore $M$
near $q$ is contained in the leaf of the Levi-foliation of $H^*$.  As $M$
is connected, the closure of the leaf must contain $p$.  The closure
of the leaf that contains $p$ must extend to a neighborhood of $p$
by Lemma~\ref{DF:lemma}, and therefore as germs at
$p$, $(M,p) \subset \Sigma'_p(H)$.

Now suppose that $M \subset (\overline{H^*}^{\mathit{rel}})_s$.  As $M$ is minimal,
it cannot be generic by Theorem~\ref{lemma:lfnmlemma}.
We write coordinates vanishing at $p$ as in Theorem~\ref{nicenormcoord}
$(z,w,w'') \in \C^n \times \C^d \times \C^k$
and define $M$ by
\begin{equation}
\begin{aligned}
\Im w & = r (z,\bar{z},\Re w) , \\
w'' & = 0 .
\end{aligned}
\end{equation}
Write $w'' = (w''_1,\ldots,w''_k)$.
If $k > 1$, then there is some affine function $L \colon \C^k \to \C$
such that $H \cap \{ L w'' = 0 \}$ is of real dimension strictly
less than $2N-2$
and therefore of dimension $2N-3$ by Lemma~\ref{lemma:interlemma}.
Thus the case $k > 1$ is finished by induction on the dimension $N$.

Therefore we are left with the case that $k=1$ (the intrinsic
complexification of $M$ is a complex hypersurface).  If
$H \cap \{ w'' = 0 \}$ is of dimension strictly less than $2N-2$ we are done
by induction as above.
Therefore assume that $\{ w'' = 0 \} \subset H$.  But then
$\{ w'' = 0 \} \subset \Sigma'_p(H)$ (as germs at $p$)
by definition of $\Sigma'_p(H)$
and we are finished.

It is left to deal with the nonminimal case.
In this case we use Theorem~\ref{nicenormcoord} to write
$M$ in the coordinates
$(z,w,w',w'') \in \C^n \times \C^{d-q} \times \C^q \times \C^k$
such that $M$ is defined by
\begin{equation}
\begin{aligned}
\Im w & = r (z,\bar{z},\Re w, \Re w') , \\
\Im w' & = 0 , \\
w'' & = 0 .
\end{aligned}
\end{equation}
Write $w' = (w'_1,\ldots,w'_q)$.  Suppose that
$H \cap \{ w'_1 = 0 \}$ is of dimension strictly less than $2N-2$.
Then $H \cap \{ w'_1 = 0 \}$ is of dimension $2N-3$
by Lemma~\ref{lemma:interlemma} and we can finish
by induction.

Thus assume that $\{ w'_1 = 0 \} \subset H$.  Then
$\{ w'_1 = 0 \} \subset \Sigma'_p(H)$ (as germs at $p$).
By Theorem~\ref{nicenormcoord} we obtain that 
$\Orb_p(M) \subset \{ w'_1 = 0 \}$ and we are done.
\end{proof}


\section{Holomorphic foliations}

A possibly singular \emph{holomorphic foliation} $\sF$ of codimension one of
a complex manifold $M$ is given by
an open covering $\{ U_\iota \}$ and a one-form $\omega_\iota$ defined
in $U_\iota$ such that if
$U_\iota \cap U_\kappa \not= \emptyset$,
then $\omega_\iota$ and $\omega_\kappa$ must be proportional at every
point of $U_\iota \cap U_\kappa$.  Furthermore $\omega_\iota$
is completely integrable, $\omega_\iota \wedge d \omega_\iota = 0$.
A complex submanifold $L \subset M$ is called a solution
if it satisfies $\omega_\iota |_{T L} = 0$
(restricted to the tangent space of $L$)
in each $U_\iota$.  The points where $\omega_\iota$ vanishes are called
the singular set of $\sF$ and denoted $\sing(\sF)$.  The set
$M \setminus \sing(\sF)$ is then a union of immersed complex hypersurfaces
called leaves of the foliation.  
The codimension
of the singularity of the foliation
can safely be taken to be at least 2, by dividing out
the coefficients of the form by any common divisors.
See \cites{CamachoNeto:book, linsneto:note}
for more information on foliations in general.

If $H$ is a nonsingular real-analytic
Levi-flat hypersurface, then the foliation of $H$ by
complex hypersurfaces, the \emph{Levi-foliation}, is a real-analytic
foliation with leaves that are complex hypersurfaces.  As locally
a real-analytic Levi-flat hypersurface can be defined by
$\{ \Im f = 0 \}$ where $df \not= 0$, we can see that the Levi-foliation
extends as a holomorphic codimension one foliation to a neighborhood of
$H$.  It is not hard to see that locally the extended foliation is uniquely
determined: if a one-form also $\omega$ defines an extension of
the foliation, then on $H$ we have $df = g \omega$ for a nonvanishing
real-analytic CR function $g$.
A real-analytic CR function on a real-analytic hypersurface uniquely
extends to a holomorphic function on a neighborhood of the hypersurface.
As $df$ and $\omega$ are proportional, they define the same unique
foliation in a neighbourhood.
We thus have the following proposition.

\begin{prop} \label{prop:gotfol}
Let $H \subset \C^N$ be a real-analytic Levi-flat submanifold of dimension
$2N-1$.  Then there exists a nonsingular codimension one holomorphic foliation
defined on a neighborhood $U$ of $H$ that extends the Levi-foliation of
$H$.
\end{prop}

A singular Levi-flat local real-hypervariety $H$ may have several components
of $H^*$ even if $H$ is irreducible.  We do, however, have the following
lemma.

\begin{lemma} \label{lemma:folonall}
Let $H \subset \C^N$ be an irreducible Levi-flat local real-hypervariety and 
$\sF$ a possibly singular codimension one holomorphic foliation
defined on a neighborhood of $H$.  Suppose that
there is
an open subset $G \subset H^*$ such that $\sF$ extends the
Levi-foliation of $G$.  Then $\sF$ extends the Levi-foliation
of $H^*$.
\end{lemma}

\begin{proof}
By analytic continuation we see that $\sF$ extends the foliation
of the whole topological component of $H^*$ that contains $G$.

Therefore, it is enough to show that if $H$ is irreducible
as a germ at some point $p \in H$ and $\sF$ extends the Levi-foliation
of some topological component $H'$ of $H^*$ such that
$p \in \overline{H'}^{\mathit{rel}}$, then $\sF$ extends the Levi-foliation of $H^*$
near $p$.  The global result then follows.

In some small neighborhood $U$ of $p$, $\sF$ is defined by a 1-form
$\omega$.  Suppose that $r(z,\bar{z})$ is the defining function for $H$
in $U$ and suppose that $H \cap U$ is irreducible.
That $\sF$ extends the Levi-foliation of $H' \cap U$ is the
same as saying that $\partial r \wedge \omega$ vanishes on $H' \cap U$.
As $H \cap U$ is irreducible, then
$\partial r \wedge \omega$ must vanish on all of $H \cap U$ and hence the
result follows.
\end{proof}

The following lemma is proved in \cite{linsneto:note}, although it is
not stated as a separate theorem.  We state the theorem in a more general
setting and so we reprove it here for completeness.  A Riemann domain over
$\C^N$ is a path-connected Hausdorff space $U$ together with a local homeomorphism
$\pi \colon U \to \C^N$.  An envelope of meromorphy of $U$ is
a Riemann domain $\widehat{U}$ such that any meromorphic function on $U$
extends to a meromorphic function on $\widehat{U}$.

\begin{lemma} \label{lemma:extendfol}
Let $U$ be a connected Riemann domain over $\C^N$, $N \geq 2$,
and let $\widehat{U}$ be the envelope of meromorphy of $U$.
Let $\sF$ be a possibly singular codimension one holomorphic foliation on
$U$.  Then $\sF$ extends to a possibly singular codimension one holomorphic
foliation on $\widehat{U}$.
\end{lemma}

\begin{proof}
The foliation $\sF$ is defined locally by completely integrable 1-forms; there exists
a covering of $U$ by open sets $\{ U_\iota \}$ and 1-forms
$\{ \omega_\iota \}$ such that $\omega_\iota = 0$
define the leaves of $\sF$.  When
$U_{\iota\kappa} = U_\iota \cap U_\kappa \not= \emptyset$, there also
exist functions $\{ h_{\iota\kappa} \}$ in $\sO^*(U_{\iota\kappa})$
such that
$\omega_\iota = h_{\iota\kappa} \omega_\kappa$ on $U_{\iota\kappa}$.
We can assume
that the codimension of the singularity of $\sF$ is 2 or greater.

The covering of $U$ can be such that $\pi$ is a homeomorphism
of $U_\iota$ onto $\pi(U_\iota)$ and so we can think of each $U_\iota$
as an open subset of $\C^N$.
We write
\begin{equation}
\omega_\iota = \sum_{j=1}^N g_j^\iota dz_j .
\end{equation}
We note that when $U_{\iota\kappa}$ is not empty then
for all $j$ we have
\begin{equation} \label{eq:compat}
g_j^\iota = h_{\iota\kappa}g_j^\kappa .
\end{equation}
As $\widehat{U}$ is connected, it follows that there
exists a $j$ such that for all $\iota$ we have
$g_j^\iota \not\equiv 0$.  We can suppose that $j=N$.

For every $j=1,\ldots,N-1$
we have meromorphic functions $f_j^\iota = 
g_j^\iota / g_N^\iota$ defined on $U_\iota$.  By \eqref{eq:compat}
on $U_{\iota\kappa}$ we have
$f_j^\iota = f_j^\kappa$ for all $j=1,\ldots,N-1$.
As $U$ is connected, for each $j=1,\ldots,N-1$, there exists a well-defined
meromorphic function $f_j$ on $U$.

Every meromorphic function on $U$
extends to a meromorphic function on $\widehat{U}$.  Thus
we have a meromorphic function $f_j$ on $\widehat{U}$
such that $f_j = f_j^\iota$ on $U_\iota$.

Now we consider the meromorphic 1-form
\begin{equation}
\eta = dz_N + \sum_{j=1}^{N-1} f_j dz_j .
\end{equation}
We can cover $\widehat{U}$ by polydiscs
$\{ \widehat{U}_\kappa \}$.
In each $\widehat{U}_\kappa$
we find a nonzero holomorphic function $\varphi_\kappa$
such that $\varphi_\kappa \eta$ has only removable singularities.
Thus we obtain a 1-form $\widehat{\omega}_\kappa$ on $\widehat{U}_\kappa$ 
that equals to $\varphi_\kappa \eta$ where that makes sense,
thus $\widehat{\omega}_\kappa$ is proportional to $\eta$ outside the
poles of the
$f_j$, and if $\widehat{U}_\kappa$ intersects
$U_\iota$, then $\eta$ is proportional to $\omega_\iota$ on $U_\iota$
outside of the poles of the $f_j$.  Therefore, $\{ \widehat{\omega}_\kappa
\}$ extend the foliation on $U$ to $\widehat{U}$.
\end{proof}

We have the following result about extending a foliation
of $H^*$.
Let $U \subset \C^N$ be a (euclidean) Hartogs figure, that is
\begin{equation}
U = \bigl(V' \times \Delta(r) \bigr)
\cup \bigl(V \times ( \Delta(r) \setminus \overline{\Delta(r')}) \bigr) ,
\end{equation}
where $V' \subset V \subset \C^{N-1}$ are two polydiscs
and $\Delta(r) \subset \C$ is a disc of radius $r$, and $0 < r' < r$.
By a theorem of Levi (see~\cites{Siu:extbook, JarnickiPflug:book})
the envelope of meromorphy of $U$ is
$\widehat{U} = V \times \Delta(r)$.
A generalized Hartogs figure $K \subset \C^N$ is a set such that
there exists a $\widehat{K} \supset K$ together with
a biholomorphic map $f \colon \widehat{K} \to \widehat{U}$,
and $f(K) = U$, where $U$ is a 
(euclidean) Hartogs figure in dimension $N$ as above.
Then $\widehat{K}$ is the envelope of meromorphy of $K$.

\begin{lemma} \label{lemma:folextendH}
Suppose that $H \subset \C^N$ is a Levi-flat local real-hypervariety
that is irreducible as germ at $p \in \overline{H^*}^{\mathit{rel}}$.
Suppose that there exists a nonsingular complex submanifold $W
\subset \overline{H^*}^{\mathit{rel}}$ of (complex) dimension at least 2,
such that there exists a generalized Hartogs figure $K \subset W \setminus
(\overline{H^*}^{\mathit{rel}})_s$ and such that $p \in \widehat{K}$.

Then there exists
a possibly singular codimension one holomorphic foliation $\sF$
extending the foliation of $H$ near $p$.
\end{lemma}

\begin{proof}
Let us take a connected component $H'$ of the
nonsingular points $(\overline{H^*}^{\mathit{rel}})_{reg}$
such that $p$
lies in the closure of $H'$, and $H'$ contains the component
of $W \setminus (\overline{H^*}^{\mathit{rel}})_s$ that contains $K$.
We define
a possibly singular codimension one holomorphic foliation
in a neighborhood $V$ of $H'$,
see Proposition~\ref{prop:gotfol}.

We can now ``fatten'' the Hartogs figure $K$ to
find a Hartogs figure $K' \subset V$ of dimension $N$.
As the envelope of meromorphy of $K'$ is $\widehat{K'}$,
we extend the foliation $\sF$ past $p$
by Lemma~\ref{lemma:extendfol}.

We have a possibly singular
holomorphic foliation $\sF$
of a neighborhood of $p$ that extends the foliation
of $H'$.  By Lemma~\ref{lemma:folonall} the foliation in fact agrees
with the foliation on all of $H^*$ (near $p$) as $H$ is irreducible at $p$.
\end{proof}

We can now prove Theorem~\ref{thm:extsmallsing}.  That is,
if the germ $(H,p)$ is irreducible and $\dim H_s < 2N-4$
or $p$ is not a leaf-degenerate point and $\dim H_s = 2N-4$, then
the Levi-foliation extends to possibly singular codimension one
holomorphic foliation in a neighborhood of $p$.

\begin{proof}[Proof of Theorem~\ref{thm:extsmallsing}]
We note that for a Riemann domain over $\C^N$, the domain
of meromorphy is a Stein manifold (See e.g.\ Theorem 3.6.6
\cite{JarnickiPflug:book}).  Therefore, it must be holomorphically convex.

Suppose that $\dim H_s \leq 2N-4$.
Take a neighborhood $U$ of $p$ in which we can apply Lemma~\ref{DF:lemma},
and such that $H \cap U$ is irreducible.  We can assume that $H$
is closed in $U$.

Let us first suppose that $N=2$, $H_s = \{ p \}$,
and $H$ is not leaf-degenerate at $p$.
Take an irreducible complex subvariety $W' \subset
\overline{H^*}^{\mathit{rel}}$
by Lemma~\ref{dimnleaf:lemma}.  Pick the 
connected topological component $H'$ of $H^*$ such that
$W' \setminus \{ p \} \subset H'$.
Define a holomorphic foliation $\sF$ on a neighborhood $\Omega$ of $H'$
by Proposition~\ref{prop:gotfol}.

As $H$ is not leaf-degenerate at $p$, then there exists
a sequence of nonsingular leaves $L_j \subset H'$
such that $p$ is a cluster point of this sequence, but such
that $p \notin L_j$ for all $j$.  By Lemma~\ref{DF:lemma},
the $L_j$ must be closed in $U$ and they must be nonsingular (as a singular
point of $L_j$ would mean a singular point of $H$).  Now consider
$L_j$ intersected with a small ball $B$ centered at $p$.
Let $K = \bigcup_j (L_j \cap \partial B)$.  
Note that $K \subset \subset \Omega$, in particular
$K$ is a positive distance away from $p$.  However,
the holomorphic hull $\widehat{K}$ (with respect to $\mathcal{O}(\Omega)$)
contains the sets $L_j \cap B$,
and $p$ is in the cluster set of the $L_j \cap B$.
If $p \in \Omega$ then we were already done.  If $p \notin \Omega$
then we see that no Riemann domain $\Omega'$ containing $\Omega$
can be Stein unless $p \in \Omega'$.  Thus the envelope of meromorphy
of $\Omega$ contains $p$ and hence a whole neighborhood of $p$.
We finish by applying Lemma~\ref{lemma:extendfol}.

When $N > 2$ and $p$ is not leaf-degenerate we proceed similarly.

If $\dim H_s < 2N-4$, then $H$ is not leaf-degenerate at any point
by Lemma~\ref{lemma:sprimeprop}.
\end{proof}

When the foliation extends, we can show that the singular set
must be Levi-flat.

\begin{lemma} \label{lemma:folandsing}
Suppose that $H \subset \C^N$ is a Levi-flat local real-hypervariety,
and
$\sF$ is a possibly singular codimension one holomorphic foliation
extending the Levi-foliation of $H$, then
$E = (\overline{H^*}^{\mathit{rel}})_s$ is Levi-flat wherever $E$ is a
CR submanifold.
\end{lemma}

\begin{proof}
We can assume that 
$E$ is a connected real-analytic CR submanifold.  Let us suppose
for contradiction that $E$ is not a Levi-flat submanifold (that of course
also means $E$ is not a complex submanifold).

Take $p \in E$.
As we are interested in $(\overline{H^*}^{\mathit{rel}})_s$, we can
without loss of generality
assume that all irreducible components of the germ $(H,p)$ are
of dimension $2N-1$.
If the foliation is nonsingular at $p$,
then it is easy to show that any $2N-1$ dimensional component of $H$ at
$p$ must be locally biholomorphic to $C \times \C^{N-1}$ for 
a one dimensional real-analytic curve $C \subset \C$.  We simply look
in a coordinate patch of the foliation where the leaves are
given by $\{ z_N = c \}$ for a constant $c$, and we note
that the leaves of the foliation must agree with leaves of $H^*$.
Therefore
$E$ must be either empty or a complex hypersurface.

Therefore suppose that the foliation $\sF$ is singular at $p \in E$.
The singular set $\sing(\sF)$ of the foliation is a complex subvariety.
As noted above
we can assume that $\sing(\sF)$ is of complex dimension $N-2$ or less and
that no point of $\sing(\sF)$ is a removable singularity.

Since at the nonsingular points $(\overline{H^*}^{\mathit{rel}})_{reg}$
the foliation must be nonsingular we see that
$\sing(\sF) \cap \overline{H^*}^{\mathit{rel}}$ must be a subset of $E$.

As we are assuming that $E$
is not a complex hypersurface, then $E = \sing(\sF) \cap
\overline{H^*}^{\mathit{rel}}$,
as where $\sF$ is nonsingular $E$ would be a complex hypersurface
or empty.  Thus for any complex hypersurface $W \subset \Sigma'_p(H)$
we obtain
\begin{equation}
E \cap W = \sing(\sF) \cap W .
\end{equation}
$\sing(\sF) \cap W$ is a complex subvariety.  We can assume that
$\Orb_p(E)$ is of maximal dimension.  By Lemma~\ref{lemma:orbinsigma}
we have that $\Orb_p(E) \subset W$ for some $W$.
Thus $\Orb_p(E) = \Orb_p(E \cap W)$ and 
$\Orb_p(E \cap W) = E \cap W$ as it is complex.
Therefore $\Orb_p(E)$ is a complex subvariety and $E$
must be Levi-flat (as $\Orb_p(E)$ was of maximal dimension).
\end{proof}


\section{Proof of the Theorem}

First a technical lemma.

\begin{lemma} \label{lemma:gethartogs}
Suppose that $V$ is a complex manifold of (complex) dimension 2 or more.
Suppose that $M \subset V$ is a real-analytic CR submanifold,
which is not Levi-flat (therefore also not complex analytic).

Then there exists a point $p \in M$ and
a generalized Hartogs figure $K \subset V \setminus M$ such that
$p \in \widehat{K}$.
\end{lemma}

\begin{proof}
This is a local theorem and hence we can assume
that $0 \in M$ and that $V$ is a small neighborhood
of the origin in $\C^k$, where we can apply Theorem~\ref{nicenormcoord}
on $M$.

Let $W \subset V$ be a 2 (complex) dimensional complex submanifold
through the origin,
we look at $M \cap W$.  If we can construct the required Hartogs figure
in $W$, then we can ``fatten'' it up to be of dimension $k$.

If the intersection $M \cap W$ is of dimension 0 or 1, then
it is not hard to construct the required Hartogs figure with $p$
being the origin.
If $M$ is not a generic submanifold, then as it
is not a complex submanifold, then we can
find such a $W$ (simply setting all variables
except $w_1$ and $w_1''$ to zero).

If $M$ is generic, then if the real codimension of $M$ is 2 or more,
we can simply set $z = 0$ and all but $2$ of the $w$ or $w'$ to zero
and we have that $M \cap W$ is a totally real submanifold of $W$.
It is then again not hard to construct the Hartogs figure with $p$
being the origin.

Hence what is left is the case when $M \subset V$ is a hypersurface.
As $M$ is not Levi-flat, then the Levi-form of $M$ is not
identically zero, then there must exist a point $p$ and
an affine complex subspace $W$ of (complex) dimension 2 such, $p \in W$
such that $M \cap W$ is strictly pseudoconvex in $W$ and
hence on the strictly pseudoconcave side of $M \cap W$ we 
can construct the Hartogs figure.
\end{proof}

Let us restate Theorem~\ref{mainthm} for reader convenience.

\begin{thmnonum}
\mainthmbody
\end{thmnonum}

\begin{proof}
Let $E = (\overline{H^*}^{\mathit{rel}})_s$.

If $N=1$, then the theorem has no content and is trivially true.
If $N=2$, then $E$ can be of dimension $1$ or $2$.
Such real subvarieties are automatically
Levi-flat near CR points (either totally-real or complex).
Hence the theorem is true
automatically for $N=2$.  From now on suppose that $N \geq 3$.

We only need to prove that near points where
$E$ is
a real-analytic CR submanifold it is Levi-flat.  Thus we can
assume without loss of generality
that $E$ is a connected real-analytic CR submanifold.

If $E$ is a generic submanifold of $\C^N$, then
by Lemma~\ref{lemma:genericsing} we have that $E$ is a generic
Levi-flat submanifold of codimension $2N-2$
and we are done.  
We thus suppose that $E$ is not a generic submanifold.

By Lemma~\ref{lemma:degensingarevar} we have that if the
set $S$ of leaf-degenerate points of $H$ is dense in $E$,
then $E$ is a complex submanifold of dimension $N-2$.
If the set $S$ is not dense in $E$, we can move to a neighborhood
of a generic point of $E$ and assume that no point of $H$
is leaf-degenerate.

Let us suppose for contradiction
that $E$ is not a complex nor a Levi-flat
submanifold.  Furthermore, by moving
to a generic point $p$ of $E$ we can assume that $\Orb_p(E)$
is of maximal possible dimension.
Near $p$ we can work in a small neighborhood $U$ of $p$
and assume that any germs of complex varieties extend to the whole
neighborhood $U$ by Lemma~\ref{DF:lemma}.  Thus as before, we 
write $\Sigma'_q(H;U)$ when we are talking
about the smallest
(closed) complex subvariety of $U$ contained in $H$ and containing $\Sigma'_q(H)$.
We also simply assume
that $H$ and $E$ are closed subsets of $U$.

If $\Sigma'_p(H;U)$ is singular, then the singular set $S$ of
$\Sigma'_p(H;U)$ would be a subset of $E$, by
Proposition~\ref{prop:nonsingsigma}.
If $(S,p) = (E,p)$ as germs,
then we are done.
We pick another point $q \in \Sigma'_p(H;U) \cap E \setminus S$.
We note that $\Sigma'_q(H)$ must contain (as a germ at $q$) a nonsingular
complex hypersurface as one of its components.  Therefore, either
$\Sigma'_q(H)$ is nonsingular as a germ, or it is reducible and
has a singularity
of complex dimension $N-2$ which must be contained in $E$.
If $\Sigma'_q(H)$ is singular (reducible) then the CR
dimension of $E$ is at least $N-2$ (as it contains
the singularity of $\Sigma'_q(H)$).  As $\Orb_q(E) \subset
\Sigma'_q(H)$, the CR dimension of $E$ is exactly $N-2$ and
$\Orb_q(E)$ is a complex submanifold of dimension $N-2$.  Because
$\Orb_q(E)$ is of maximal dimension, $E$ must be Levi-flat.

Therefore let us assume that the germ $\Sigma'_p(H;U)$ is nonsingular,
and in fact we can assume that the germ $\Sigma'_q(H)$ is nonsingular
as germ at $q$ for all $q \in \Sigma'_p(H;U)$.

As we are assuming $\Orb_p(E)$ is of
maximal dimension, we can pick coordinates vanishing at $p$ as in
Theorem~\ref{nicenormcoord}.  Furthermore if
$E \subset \Sigma'_p(H;U)$, then we can pick coordinates
such that $\Sigma'_p(H) = \{ w''_1 = 0 \}$ as germs at $p$.

If $E \not\subset \Sigma'_p(H;U)$,
then as $\Orb_p(E) \subset \Sigma'_p(H)$ by
Lemma~\ref{lemma:orbinsigma}, we can pick coordinates
such that $\Sigma'_p(H) = \{ w'_1 = 0 \}$ as germs at $p$.

In either case, if $E$ was not Levi-flat, then $E \cap \Sigma'_p(H;U)$
is not Levi-flat.
We now appeal to Lemma~\ref{lemma:gethartogs} to obtain
a Hartogs figure $K$ inside $\Sigma'_p(H;U)$.
To do so,
we may have needed
to perhaps move to yet another point $p' \in E \cap \Sigma'_p(H;U)$.
This move is allowed as we are assuming that $\Sigma'_{p'}(H)$ is
nonsingular.

As $E$ is a connected
real-analytic CR submanifold that is neither
generic Levi-flat nor complex analytic,
we can apply Proposition~\ref{prop:tech}
and assume that $H$ is irreducible at $p$.
We can now
appeal to Lemma~\ref{lemma:folextendH} to obtain
a foliation $\sF$ near $p$.  Next we appeal to
Lemma~\ref{lemma:folandsing} to get a contradiction ($E$ is Levi-flat
though we assumed it was not).
\end{proof}


\def\MR#1{\relax\ifhmode\unskip\spacefactor3000 \space\fi%
  \href{http://www.ams.org/mathscinet-getitem?mr=#1}{MR#1}}

\end{document}